\documentclass[10pt]{amsart}

\usepackage[utf8]{inputenc}
\usepackage[german,english]{babel}     
\usepackage{amsmath}
\usepackage{amsfonts}
\usepackage{amssymb}
\usepackage{amsthm}
\usepackage{mathrsfs}                  
\usepackage{appendix}                  
\usepackage{tikz}                      
\usetikzlibrary{arrows}

\newcommand{\lra}{\longrightarrow}

\newcommand{\C}{{\mathbb{C}}}

\newcommand{\HH}{{\mathbb{H}}}         
\newcommand{\I}{{\mathbb{I}}}          

\newcommand{\Q}{{\mathbb{Q}}}
\newcommand{\R}{{\mathbb{R}}}

\newcommand{\Z}{{\mathbb{Z}}}          

\newcommand{\Dd}{{\mathcal{D}}}        

\newcommand{\Hh}{{\mathcal{H}}}        

\newcommand{\Ss}{{\mathcal{S}}}        

\newcommand{\swann}{\mathcal{U}(N)}    
\newcommand{\euler}{\mathcal{X}_{0}}   

\newcommand{\sst}{\scriptscriptstyle}
\newcommand{\mc}[1]{\mathcal{#1}}
\newcommand{\mr}[1]{\mathrm{#1}}
\newcommand{\ms}[1]{\mathsf{#1}}
\newcommand{\scr}[1]{\mathscr{#1}}
\newcommand{\mf}[1]{\mathfrak{#1}}

\newcommand{\eqst}[1]{\begin{equation*} #1 
                      \end{equation*}}
\newcommand{\eq}[1]{\begin{equation} #1
                    \end{equation}}
\newcommand{\alst}[1]{\begin{align*} #1
                      \end{align*}}
\newcommand{\al}[1]{\begin{align} #1
                      \end{align}}

\theoremstyle{plain}
\newtheorem{thm}{Theorem}[section]
\newtheorem{lem}[thm]{Lemma}
\newtheorem*{lemma*}{Lemma}
\newtheorem{prop}[thm]{Proposition}

\theoremstyle{definition}
\newtheorem{defn}{Definition}

\newtheorem{ex}{Example}
\newtheorem*{note}{Note}
\newtheorem{rmk}{Remark}

\theoremstyle{remark}
\newtheorem*{question*}{Question}

\DeclareMathOperator{\id}{id}
\DeclareMathOperator{\End}{End}
\DeclareMathOperator{\Map}{Map}
\DeclareMathOperator{\Hom}{Hom}
\DeclareMathOperator{\pr}{pr}
\DeclareMathOperator{\vl}{vl}
\DeclareMathOperator{\im}{im}

\usepackage{hyperref}
\hypersetup{
    colorlinks = true,%
    linkcolor = blue,%
    citecolor = blue,%
    pdfcreator={LaTeX},
    pdftitle={Generalized spinors hyperKahler geometry and kahler structures on 4-manifolds},%
    pdfsubject={hyperkahler manifolds, kahler structures, 4-manifolds},%
    pdfauthor={Varun Thakre},%
    pdfkeywords={gauge theory, generalized dirac operator, non-linear dirac operato, hyperkahler manifolds, kahler structures},%
}

\begin{document}

\title{K\"ahler and symplectic structures on 4-manifolds and hyperK\"ahler geometry}

\author{Varun Thakre}

\curraddr{Mathematics dept., Harish-Chandra Research Institute, Allahabad, India}
\email{varunthakre@hri.res.in}


\date{\today}



\begin{abstract}
A non-linear generalization of the Dirac operator in 4-dimensions, obtained by replacing the spinor representation with a hyperK\"ahler manifold admitting certain symmetries, is considered. We show that the existence of a covariantly constant, generalized spinor defines a K\"ahler structure on the base 4-dimensional manifold. For a class of hyperK\"ahler manifolds obtained via hyperK\"ahler reduction, we also show that a harmonic spinor, under mild conditions, defines a symplectic structure. Finally, we show that if a covariantly constant, generalized spinor satisfies generalized Seiberg-Witten equations, the metric on the base manifold has a constant scalar curvature.
\end{abstract}

\maketitle


\section{Introduction}

 
 Let $(X, g_{\sst X})$ be a 4-dimensional, smooth, oriented Riemannian manifold and let $\pi:Q \lra X$ be a $Spin^c$-structure on $X$. The Levi-Civita connection on the frame bundle $\pi_{SO}:P_{SO(4)} \lra X$ and a connection $A$ on the principal $U(1)$-bundle $\pi_{U(1)}:P_{U(1)}:=Q/Spin(4) \lra X$ determine a unique connection $\ms{A}$ on $Q$. Let $W^{\pm} \lra X$ denote the associated positive and negative spinor bundles respectively and $\sigma: W^{+} \lra \Lambda^2_+(X)$ denote the quadratic map. If $u \in \Gamma(X, W^+)$ is a non-vanishing section, then $\sigma\circ u$ defines a non-degenerate, self-dual 2-form on $X$. It was shown in \cite{BLPR2000, scorpan2002} independently that if $\nabla^{\ms{A}} u =0$ for some $U(1)$-connection $A$, then $\sigma\circ u$ defines a K\"ahler structure on $X$, compatible with a metric $g'_{\sst X} = c\cdot g_{\sst X}, ~~ c\in\R$. On the other hand, under mild conditions, it was shown that if $u$ is a harmonic spinor, then $\sigma\circ u$ defines a symplectic structure on $X$. Scorpan \cite{scorpan2002} gave a charaterization of the K\"ahler and symplectic 2-forms that lie in the image of the quadratic map.
 
Taubes introduced a non-linear generalization \cite{taubes} of the $Spin$-Dirac operator in dimension 3, wherein the spinor representation is replaced by a hyperK\"ahler manifold $(M, g_{\sst M}, I_{1}, I_{2}, I_{3})$ - also known as the \emph{target hyperK\"ahler manifold} - admitting an action of $Sp(1)$ that permutesthe 2-sphere of complex structures. \emph{Generalized spinors} are defined to be sections of the associated fibre-bundle with a typical fibre $M$. The Dirac operator is replaced by a first-order, non-linear elliptic differential operator $\Dd$ for maps taking values in $M$. For a twisting principal $G$-bundle $P_G$, every connection $A$ on $P_G$ defines a twisted, generalized Dirac operator $\Dd_{\ms{A}}$. The idea was extended to dimension 4 by Pidstrygach \cite{victor}. 

The current article investigates the role of certain special generalized spinors and hyperK\"ahler manifolds in defining a K\"ahler or a symplectic structure on $X$.

One way of obtaining hyperK\"ahler manifolds with requisite properties, is via Swann's construction \cite{swann}. Starting with a quaternionic K\"ahler manifold of positive scalar curvature, Swann's construction produces a hyperK\"ahler manifold endowed with a permuting $Sp(1)$-action. The manifold is a fibration over the quaternionic K\"ahler manifold. Additionally an action of a Lie group $G$ that preserves the quaternionic K\"ahler structure can be lifted to a hyperHamiltonian action on the hyperK\"ahler manifold. For target hyperK\"ahler manifolds obtained via Swann's construction, covariantly constant spinors define a K\"ahler structure on $X$:

\begin{thm}
 \label{thm:Swann bundles kahler structure}
Let ~$\swann$ denote the total space of a Swann bundle over some quaternionic K\"ahler manifold $N$ with positive scalar curvature and assume that $N$ admits an action of $U(1)$ that preserves the quaternionic K\"ahler structure. Let $\mu: M \lra \mf{sp}(1)^{\ast}$ denote the associated hyperK\"ahler $U(1)$-moment map and $u \in \Map(Q, ~\swann)^{Spin^c}$ be a spinor whose range does not contain a fixed point of $U(1)$-action on $\swann$. If there exists a connection $\ms{A}$ on $Q$ such that the covariant derivative $D_{\ms{A}}u = 0$, then, under the isomorphism $\Phi: \mf{sp}(1)^{\ast} \rightarrow \Lambda^2_+(\R^4)^{\ast}$, $\omega:= \Phi(\mu \circ u)$ defines a K\"ahler structure on $X$.
\end{thm}

 Another way of constructing such hyperK\"ahler manifolds is via hyperK\"ahler reduction of $\HH^{n}$ by a hyperHamiltonian action of a Lie group $H$. The technique due to Hitchin, Karlhede, Lindstr\"om and Roc\v{e}k \cite{hklr}, is an analogue of Marsden-Weinstein reduction for symplectic manifolds. It has proven to be quite useful in constructing highly non-trivial hyperK\"ahler manifolds, starting from flat quaternionic spaces. If the usual action of $Sp(1)$ on $\HH^n$ preserves the zero set $\mu^{-1}_H(0)$ of the $H$-moment map, then the action descends to a permuting action on the quotient $M_H:=\mu^{-1}_H(0)/H$. Moreover, if there exists a hyperHamiltonian $U(1)$-action on $\HH^n$ that commutes with $H$-action and preserves $\mu^{-1}_H(0)$, then it descends to a hyperHamiltonian action on $M_H$.
 
For $M = M_H$, the covariantly constant spinors define a K\"ahler structure on $X$:
\begin{thm}
 \label{thm:HK-reduction kahler structure}
Let $\mu: M_H \lra \mf{sp}(1)^{\ast}$ denote a hyperK\"ahler $U(1)$-moment map and $u \in \Map(Q, M_H)^{Spin^c}$ be a spinor whose range does not contain a fixed point of $U(1)$-action on $M$. If there exists a connection $\ms{A}$ on $Q$ such that the covariant derivative $D_{\ms{A}}u = 0$, then, $\omega:= \Phi(\mu \circ u)$ defines a K\"ahler structure on $X$.
\end{thm} 

Morever, in this case, the harmonic spinors, under mild conditions, define a symplectic structure on $X$:
\begin{thm}
\label{thm:HK-reduction symplectic structure}
Let $\ms{A},~\mu$ be as in Theorem \ref{thm:HK-reduction kahler structure}. Let $u \in \Map(Q, M_H)^{Spin^c}$ be a spinor whose range does not contain a fixed point of $U(1)$-action on $M_H$. Assume that $D_{\ms{A}}u \perp  \ker{d\mu}$. If $\Dd_{\ms{A}} u = 0$, then $\omega$ defines a symplectic structure on $X$.
\end{thm} 

\vspace{0.3cm}  
 The layout of the article is as follows: Section \ref{sec:preliminaries} is divided into two parts. In the first part, sub-section \ref{subsec:hyperkahler manifolds}, we introduce the preliminaries on hyperK\"ahler manifolds and describe the two mentioned constructions; namely hyperK\"ahler reduction and Swann's construction. In the second part, sub-section \ref{subsec:generalized dirac operator}, we introduce the preliminaries needed in order to define the generalized Dirac operator. The details of some of the technical part in this section is left to the Appendix. Section \ref{sec:kahler and symplectic structures}, gives the proof of Theorem \ref{thm:Swann bundles kahler structure}, Theorem \ref{thm:HK-reduction kahler structure} and Theorem \ref{thm:HK-reduction symplectic structure}. Finally in Section \ref{sec:equivalence of structures}, we prove that the $Spin^c$-structures defined by $\omega$ in both the situations is isomorphic to the one determined by $g_{\sst X}$.
 

\subsection*{Acknowledgements} The ideas used in the second half of this article are a part of the author's doctoral dissertation. The author wishes to thank his advisor Prof. Dr. Viktor Ya. Pidstrygach for his encouragement and endless support. The author wishes to thank DFG (Deutsche Forschungsgemeinschaft) for financial support during his tenure as a doctoral student at Georg-August-Universit\"at, G\"ottingen.
 

 \section{Preliminaries and Notations}
 \label{sec:preliminaries} 
 

 \subsection{HyperK\"ahler manifolds with permuting actions}
 \label{subsec:hyperkahler manifolds}
 

A $4n$-dimensional Riemannian manifold $(M, g_{\sst M})$ is said to be \emph{hyperK\"ahler} if it is endowed with a set of almost-complex structures 
\eqst{
I_i \in \End(TM), ~~ i=1,2,3, ~~ I_i I_j = \delta_{ijk} I_k
}
that are covariantly constant with respect to the Levi-Civita connection. The quaternionic structure induces a covariantly constant algebra homomorphism
 \eq{
 \label{eq:algebra homomorphism}
 \I: \HH \lra \End(TM), ~~~ \I_\xi := \I(\xi) = \xi_0 \id_{TM} + \xi_1 I_1 + \xi_2 I_2 + \xi_3 I_3 ~~ \text{for} ~~ \xi \in \HH
 }
Sitting inside the quaternion algebra, is the standard 2-sphere of purely imaginary quaternions
\eqst{
S^2 = \{ \xi = \xi_1 i + \xi_2 j + \xi_3 k ~|~ |\xi| = 1   \}
}
 Every $\xi \in S^2$ defines a K\"ahler structure on $M$. In other words, $M$ has an entire family of K\"ahler structures parametrized by $S^2 \in \mf{Im}(\HH)$. 
 
\begin{defn}
An isometric action of a smooth Lie group $G$ on $M$ is said to be \emph{tri-holomorphic} if it fixes the 2-sphere of complex structures $S^2$; i.e,
\eqst{ 
 Tg ~ \I_{\xi} = \I_{\xi} ~ Tg, ~~ \text{for} ~~ g \in G,~~ \xi \in S^2
}
\end{defn}
In particular, $G$ preserves the K\"ahler 2-forms $\omega_{i} = g_{\sst M}(I_i(\cdot), \cdot)$, for $i=1,2,3$. We can combine $\omega_i$ to define a single $\mf{sp}(1)$-valued 2-form 
\eqst{
\omega \in \mf{sp}(1)^{\ast}\otimes\Lambda^2 M, ~~ \omega_{\xi}:=\langle \omega, \xi \rangle = g_{\sst M}(\I_{\xi}(\cdot), \cdot)
}
Additionally, if all the three associated moment maps exist, then the action is said to be \emph{tri-Hamiltonian}. Again, one can combine the three moment maps into one to define a  $\mf{sp}(1)$-valued map $\mu: M \longrightarrow \mf{sp}(1)^{\ast}\otimes \mf{g}^{\ast}$, that satisfies
\begin{enumerate}
\item $d\mu = \iota_{\mf{g}} \omega$ 
\item $\mu(gh) = \text{Ad}^{\ast}_g (\mu(h))$
\end{enumerate}

The map $\mu$ is called a \emph{hyperK\"ahler moment map}.

\begin{ex}
Consider $\HH^n$ with the hyperK\"ahler structure given by $(R_{\bar{i}}, R_{\bar{j}}, R_{\bar{k}})$. Let $Sp(n)$ denote the group of $\HH$-linear isometries of $\HH^n$ and $G \subset Sp(n)$.

The $G$-action on $\HH^n$, given by left multiplication $G \times \HH^n \ni (g, h) \longmapsto gh \in \HH^n$ is a tri-holomorphic action with the hyperK\"ahler moment map 
\eqst{
\langle \mu(x), \xi \otimes \eta \rangle = \frac{1}{2} \xi x^{\dagger}\eta x
}
\end{ex}

\subsection*{HyperK\"ahler reduction}
\label{subsec:hyperkahler reduction}

Many non-trivial examples of hyperK\"ahler manifolds can be constructed via hyperK\"ahler reduction, from the quaternionic vector space $\HH^n$. HyperK\"ahler reduction is an extension of the well-known Marsden-Weinstein reduction for symplectic manifolds. 

Suppose that $M$ is endowed with a hyper-Hamiltonian action of a compact Lie group $H$. Let $\mu_H: M \lra \mf{sp}(1)^{\ast}\otimes \mf{h}^{\ast}$ be a hyperK\"ahler moment map for the $H$-action and assume that $a \in \mf{h}$ is invariant under the co-adjoint action of $H$. Then $\mu^{-1}_H(a)$ is an $H$-invariant submanifold of $M$. 

\begin{thm}\cite{hklr}
Let $a \in \mf{h}$ be a central regular value of the moment map $\mu_H$ and assume that $H$ acts freely and properly on $\mu^{-1}_H(a)$. Then the quotient $M_H:=\mu^{-1}_H(a)/H$ is again a hyperK\"ahler manifold.
\end{thm}

\begin{ex}
Many interesting examples of hyperK\"ahler manifolds fall under this category. To name a few
\begin{enumerate}
\item Co-tangent bundles of complex Lie groups \cite{kron88, kobak-swann96}
\item Co-adjoint orbits of semi-simple Lie groups \cite{kronheimer}
\item Moduli space of framed, charge $k$ instantons on $S^4$ \cite{adhm}
\end{enumerate}
\end{ex}

\subsection*{Swann bundles}

Let $Sp(1)$ denote the group of unit quaternions and $\mf{sp}(1)$ its Lie algebra. Note that $\mf{sp}(1) \cong \mf{Im}(\HH)$. A \emph{permuting action} of $Sp(1)$ or $SO(3)$ on $M$ is an isometric action, such that the induced action on the 2-sphere of complex structures $S^2$ is the standard action of $Sp(1)$ or $SO(3)$ on $S^2$
\eqst{ 
 Tq ~ \I_{\xi} ~ Tq^{-1} = \I_{ q\xi \bar{q}}, ~~ \text{for} ~~ q \in Sp(1),~~ \xi \in S^2 \subset \mf{sp}(1)
 } 

Amongst the hyperK\"ahler manifolds admitting a permuting $Sp(1)$-action, there are those that also admit a \emph{hyperK\"ahler potential}. A hyperK\"ahler potential is a real-valued function $\rho: M \lra \R$ which is a K\"ahler potential w.r.t all three complex structures simultaneously; i.e,
\eqst{
-d(\I_{\xi}(d\rho)) = 2\omega_{\xi}~~ \text{for}~~ \xi \in \mf{sp}(1), ~~ ||\xi||^2 = 1.
}

A \emph{quaternionic K\"ahler manifold} $N$ is a $4n$-dimensional manifold whose holonomy is contained in $Sp(n) Sp(1)  := (Sp(n) \times Sp(1))/\pm 1$. Let $F_{N}$ denote the reduction of the principal frame bundle $P_{SO(4n)}$ to $Sp(n) Sp(1)$-bundle over $N$. Then $\mathscr{S}(N) = F_{N}/Sp(n)$ is a principal $SO(3)$-bundle, which is a frame bundle of the three dimensional vector subbundle of skew symmetric endomorphisms of $TN$. The $Sp(1)$-action on $\HH$ (by left multiplication) descends to an isometric action of $SO(3)$ on $\HH^{*}/\pm 1$. The Swann bundle over $N$ is defined to be the principal $\HH^{*}/\Z_{2}$-bundle 
\eqst{ 
\mathcal{U}(N):=\mathscr{S}(N) \times_{\scriptscriptstyle SO(3)} \HH^{*}/\Z_{2} \lra N
} 

\begin{thm}[\cite{swann}]
\label{thm: swanns theorem}
Let $N$ be a quaternionic K\"ahler manifold with a positive scalar curvature. Then $\swann$ is a hyperK\"ahler manifold with a free, permuting $Sp(1)$ action. Moreover, $\swann$ also admits a hyperK\"ahler potential given by \[\rho_{0} = \frac{1}{2} r^{2}\] 
where $r$ is the co-ordinate along $\HH^{\ast}/\Z_2$. If $N$ has an isometric action of a Lie group $G$ that preserves the quaternionic K\"ahler structure, then the action can be lifted to a hyper-Hamiltonian action of $G$ on $\swann$.
\end{thm}

The hyperK\"ahler potential on $\swann$ is quite special. Namely, if $\euler := \text{grad} \rho_0$ - also known as \emph{Euler vector field} - the fundamental vector fields due to permuting $Sp(1)$-action satisfy:
 \eq{
 \label{eq:hyperkahler potential}
 \I_{\xi}K^M_{\xi} = -\euler, ~~~ \xi \in S^2 \subset \mf{sp}(1) ~~~\text{and}~~~ \rho_0 = \frac{1}{2} g_{\sst M}(\euler, \euler)
 }
 
The moment map for a hyper-Hamiltonian $G$-action on $\swann$ has a simple form \cite[Corollary 3.3.1]{henrik}
 \eq{
 \label{eq:moment map swann bundle}
 \langle \mu, \xi\otimes\eta \rangle = -\frac{1}{2}g_{\sst M}\left(K^M_{\xi}, K^M_{\eta}\right), ~~ \xi \in \mf{sp}(1), ~~ \eta \in \mf{g}
 }

\begin{ex}
The flat space $\HH^{n^{\ast}}:= \HH^n \setminus \{0\}$ is the total space of Swann bundle over $\HH P^{n-1}$. Indeed, observe that $\HH^{n^{\ast}} = Sp(n) \times \R^{+}$. A permuting $Sp(1)$-action on $\HH^n$ is given by 
\eqst{
Sp(1) \times \HH^n \ni (q, h) \longmapsto h\bar{q} \in \HH^n
}
The Euler vector field $\euler = \id_{\HH^n}$ and therefore, the hyperK\"ahler potential $$\displaystyle \rho_0(h) = \frac{1}{2} ||h||^2.$$
\end{ex}

On the other hand, consider a hyperK\"ahler manifold $M$ endowed with a permuting $Sp(1)$-action.

\begin{thm}[\cite{swann}]
If $M$ admits a hyperK\"ahler potential $\rho_0: M \longrightarrow \R^+$, then for $c \in \R$, $N := \rho^{-1}_0(c)/Sp(1)$ is a quaternionic K\"ahler manifold of positive scalar curvature. Consequently, $M$ is the total space of a Swann bundle over $N$.
\end{thm}

Define $Spin^G_{\varepsilon}(4):= Spin(4)\times_{\Z_2} G$, where $\Z_2$ denotes an order 2-subgroup generated by $(-1,\varepsilon)$ with the central element $\varepsilon \in G$. An action of $Spin^G_{\varepsilon}(4)$ is said to be permuting if the action of $Sp(1)_{+}\hookrightarrow Spin^G_{\varepsilon}(4)$ is permuting and the action of $Sp(1)_{-} \times G \hookrightarrow Spin^G_{\varepsilon}(4)$ is tri-holomorphic. For the rest of the article, we will study the case when $G = U(1)$.



\subsection{Generalized Dirac operator}
\label{subsec:generalized dirac operator}


Fix a $Spin^c$-structure $\pi: Q \longrightarrow X$. The Levi-Civita connection $\phi$ on the frame bundle $P_{SO(4)}$ and a connection $A$ on $P_{U(1)}$ uniquely define a connection $\ms{A}$ on $Q$. Let $\scr{A} \subset \Lambda^{1}(Q, \mf{spin}(4))^{Spin^c}$ denote the space of all connections which are the lift of the Levi-Civita connection. Let $M$ be a manifold admitting a permuting action of $Spin^c(4)$.  We define the space of \emph{generalized spinors} to be the space of smooth, equivariant maps $\Ss:= \Map(Q, M)^{Spin^c} \cong \Gamma(X, Q \times_{Spin^c}M)$.

The covariant derivative of a spinor $u \in \Ss$, w.r.t $\ms{A} \in \scr{A}$ is defined as
 \eq{
 \label{covariant derivative}
 \begin{split}
  D_{\ms{A}} : C^{\infty}(Q,M)^{Spin^c} & \lra \Hom(TQ, TM)^{Spin^c}_{hor} \cong C^{\infty}(Q, (\R^{4})^{\ast}\otimes TM)^{Spin^c} \\
 & D_{\ms{A}} u = du + K^M_{\ms{A}}|_{u}
  \end{split}
  }
where $K^M_{\ms{A}}|_{u}: TQ \rightarrow u^{\ast}TM$ is vector bundle homomorphism
\eqst{
K^M_{\ms{A}}|_{u} (v) = K^M_{A(v)}|_{u(p)}, ~~~ v\in T_p P
}
Alternatively, one can view the covariant derivative as
\al{
\label{eq:covariant derivative alt. def.}
D_{\ms{A}}:&~ C^{\infty}(Q,M)^{Spin^c} \lra C^{\infty}(Q, (\R^{4})^{*} \otimes TM)^{Spin^c} \\
&\langle D_{\ms{A}}u(p), w \rangle = du(p)(\tilde{w})
}
where, $w \in \R^{4}$, $\tilde{w}$ denotes the horizontal lift of $\pi_{\sst SO}(p)(w) \in T_{\pi(p)}X$.

Let $\psi: TTM \lra TM$ denote the Levi-Civita connector on $M$ and $\nabla^{\ms{A}, \psi}$, the linearization of $D_\ms{A}$ (see Appendix \eqref{principal bundles and covariant derivatives on associated bundles}, Lemma \eqref{linearization cov. der.})
\eq{
\label{linearization}
\nabla^{\ms{A}, \psi} : C^{\infty}(Q, TM)^{Spin^c} \lra \Hom(TQ, TM)^{Spin^c}_{hor}
 }
 The following Lemma is crucial to our construction in the next section
\begin{lem}\cite[Corollary 4.6.2]{henrik}
\label{linearized cov. der}
Let $M$ be a manifold with a permuting $Sp(1)$-action and a hyperK\"ahler potential $\rho_0$ and let $\euler: = \text{grad}~\rho_0$. Let $u \in \Ss$. Then $\euler\circ u \in \Gamma(Q, u^{\ast}TM)^{Spin^c}$. For a connection $\ms{A}$ on $Q$,
 \eq{
 D_{\ms{A}} u = \nabla^{\ms{A}, \psi} (\euler \circ u)
 }
\end{lem}

\subsection*{Clifford multiplication}
 Define $W^+$ to be the $Spin^c$-equivariant bundle $TM \lra M$ equipped with an action induced by $\vartheta_+ = [q_{+} q_{-}, g] \longmapsto [q_{+},g]$. More precisely, for any $w_{+} \in W^+$, the action is given by: 
 \eqst{
 [q_{+}, q_{-}, g] \cdot w_{+} = Tq_{+} Tg w_{+}.
 }
 Define the $W^-$ to be the $Spin^c$-equivariant bundle $TM \lra M$ equipped with the following action:
\eqst{
 [q_{+}, q_{-}, g] \cdot w_{-} = \I_{q_{-}} \I_{\bar{q}_{+}} Tq_{+} Tg ~ w_{-}
 }
 
Now Clifford multiplication is is a map of $Spin(4)$-representations
\eqst{
\mf{m}: \R^4 \lra \End(W^+ \oplus W^-)
}
We identify $\R^{4}$ with $\HH$ by mapping the standard, oriented basis  $(e_{1},e_{2},e_{3},e_{4})$ of $\R^{4}$, to $(1, \bar{i}, \bar{j}, \bar{k})$. The hyperK\"ahler structure on $\HH$ is given by $(R_{\bar{i}}, R_{\bar{j}}, R_{\bar{k}})$ and the $Spin^c$ action on $\HH$ by $[q_{+}, q_{-}, g] \cdot h = q_{-}h \bar{q}_{+}$. Define the map
 \alst{
  \mathfrak{m}: & \R^{4} \cong \Q \lra \End(W^+ \oplus W^-) \\
 & h \longmapsto \begin{bmatrix} 0 & -\I_{\bar{h}} \\ \I_{h} & 0  
 \end{bmatrix}
  }
Since $ \mf{m}(h)^{2} = -g_{\R^{4}}(h,h)\cdot id_{W^+ \oplus W^-}$, by universality property,  extends to a map of algebras $\tilde{\mf{m}}: \mc{C}l_{4} \lra \End(W^+ \oplus W^-)$. Identifying $\R^{4}~~ \text{with} ~~(\R^{4})^{*}$ we define Clifford multiplication by:
\begin{center}
$\bullet:(\R^{4})^{*} \otimes (W^+ \oplus W^-) \lra  W^+ \oplus W^-$\\
$g_{\R^{4}}(h, \cdot) (w_{+}, w_{-}) \longmapsto \mathfrak{m}(h)(w_{+}, w_{-})$.
\end{center}
This map is $Spin^c$-equivariant \cite{henrik}.

Composing Clifford multiplication $\mf{c}$ with the covariant derivative, we get the \emph{non-linear Dirac operator}:
\eq{
\Dd_{\ms{A}} u \in C^{\infty}(Q, u^{*}W^-)^{Spin^c}
}

More explicitly, from \eqref{eq:covariant derivative alt. def.}, we get
\eq{
\label{eq:gen. dirac operator explicit}
\Dd_{\ms{A}}(u) = \sum_{i=0}^{3} e_{i} \bullet D_{\ms{A}}u(\tilde{e_{i}})
}

Let $\mf{c}_{u}$ be the restriction of the Clifford multiplication $\mf{c}$ to $u^{\ast}W^+ \oplus u^{\ast}W^-$. Consider the first-order differential operator 
\eq{
\label{linearized dirac op}
\Dd^{lin}_{\ms{A}, u} := \mf{c}_u \circ \nabla^{\ms{A}, \psi}: u^{\ast}W^+ \oplus u^{\ast}W^- \rightarrow u^{\ast}W^+ \oplus u^{\ast}W^-
}

\begin{lem}\cite[Lemma 4.6.1]{henrik}
The linearization of the generalized Dirac operator, at a point $(u, A) \in \Ss \times \scr{A}$, coincides with the linear operator $\Dd^{lin}_{\ms{A}, u}$.
\end{lem}

We are now in a position to state the Weitzenb\"ock formula
\begin{thm}\cite[Theorem 4.7.1]{henrik}
\label{weitzenbock formula}
Weitzenb\"ock formula for the generalized Dirac operator:
\eq{
\Dd^{lin, \ast}_{\ms{A}, u}\Dd_{\ms{A}}u = \nabla^{\ms{A}, \psi, \ast}D_{\ms{A}} u + \frac{s_{\sst X}}{4} \euler\circ u + \mc{Y}(F^+_{A})|_{u}
}
where $\Dd^{lin, \ast}_{\ms{A}, u}: T_u \Ss \rightarrow W^-$ is the adjoint of the linearized Dirac operator, $s_{\sst X}$ is the scalar curvature of $X$ and the vector field $\mc{Y}(F^+_{A})|_{u} = \displaystyle \sum^3_{l=1} I_l K^M_{\sst \langle F^+_{A}, \xi_l \rangle}|_{u}$ where $\{\xi_l\}$ is the basis of $\mf{sp}(1)$.
\end{thm}

\subsection*{Generalized Seiberg-Witten equations}
Let $\mu$ denote hyperK\"ahler moment map for the hyperK\"ahler $U(1)$-action on $M$ and $F_{\ms{A}} \in \Map(Q, \Lambda^2(\R^4)^{\ast}\otimes\mf{sp}(1))^{Spin^c}$ denote the curvature of the connection $\ms{A}$. The \emph{generalized Seiberg-Witten equations} for a pair $(u, \ms{A}) \in \Ss \times \scr{A}$, in dimension 4, are
\begin{equation}
\label{eq:gen. seiberg-witten}
\left\{\begin{array}{l}
\Dd_{\ms{A}} u = 0 \\
F^{+}_{\ms{A}} - \mu \circ u = 0
\end{array}\right.
\end{equation}


\section{K\"ahler structure on 4-manifolds and Swann bundles}
\label{sec:kahler and symplectic structures} 
In this section, we give the proof of Theorem \ref{thm:Swann bundles kahler structure}. Assume that $\swann$ is endowed with a hyper-Hamiltonian action of $U(1)$ and let $\mu: M \lra \mf{sp}(1)^{\ast}$ denote the associated moment map \eqref{eq:moment map swann bundle}. Observe that since the $Sp(1)$ action on $\swann$ is free, $\euler \circ u\in \Gamma(Q, u^{\ast}W^+)^{Spin^c}$ is a non-vanishing section.

\begin{proof}[Proof of Theorem \ref{thm:Swann bundles kahler structure}]

Let $(u, \ms{A}) \in \Ss \times \scr{A}$ be such that $D_{\ms{A}} u = 0$ and the image of $u$ does not contain any fixed points of the $U(1)$-action. This has the consequence that $d(\mu \circ u) = \langle d\mu, D_{\ms{A}} u \rangle = 0$, which implies that $\mu \circ u$ is constant. Under the isomorphism $\Phi: \mf{sp}(1)^{\ast} \rightarrow \Lambda^2_+(\R^4)^{\ast}$, the map $\mu \circ u: Q \lra \mf{sp}(1)$ defines a non-degenerate, self-dual 2-form $\omega:= \Phi(\mu\circ u)$ on $X$. We treat $\omega$ as an element in $\Map(Q, \Lambda^2_+(\R^4)^{\ast})^{Spin^c}$. 
\al{
\label{eqtn. 1}
\nabla^{\ms{A}, \psi}(\omega\bullet (\euler \circ u))
& = \left(D_{\phi} \omega \right)\bullet \euler\circ u + \omega \bullet \nabla^{\ms{A}, \psi}(\euler \circ u) \nonumber  \\
& = \left(D_{\phi} \omega \right)\bullet \euler\circ u + \omega \bullet D_{\ms{A}} u \nonumber  \\
& = \left(D_{\phi} \omega \right)\bullet \euler\circ u\\ \nonumber
}
 
Consider the left hand side of \eqref{eqtn. 1}
 \alst{ \omega\bullet (\euler \circ u)
 &= (\mu_1 \circ u) \cdot \left(e_0\wedge e_1 + e_2 \wedge e_3 \right) \bullet \euler\circ u \\
 &+ (\mu_2 \circ u) \cdot \left(e_0\wedge e_2 - e_1 \wedge e_3 \right) \bullet \euler\circ u \\
 &+ (\mu_3 \circ u) \cdot \left(e_0\wedge e_3 + e_1 \wedge e_2\right) \bullet \euler\circ u \\
 &= (\mu_1 \circ u) \cdot \left((e_0 \cdot e_1) \bullet \euler\circ u + (e_2 \cdot e_3) \bullet \euler\circ u\right) \\
 &+ (\mu_2 \circ u) \cdot \left((e_0 \cdot e_2) \bullet \euler\circ u - (e_1 \cdot e_3) \bullet \euler\circ u\right)\\
 &+ (\mu_3 \circ u) \cdot \left((e_0 \cdot e_3) \bullet \euler\circ u + (e_1 \cdot e_2) \bullet \euler\circ u\right)\\
 &= \sum^3_{l=1}(\mu_l \circ u)~I_l ~ (\euler\circ u)
  } 
 
Since $\mu \circ u$ is constant, this implies that $\mu_l \circ u$ are constant for all $l = 1,2,3$. Therefore the left hand side of the eq. \eqref{eqtn. 1} reads
\al{
\label{eqtn. 2}
\nabla^{\ms{A}, \psi}(\omega\bullet \euler \circ u) 
&= \nabla^{\ms{A}, \psi} \left(\sum^{3}_{l=1} (\mu_l \circ u) ~ I_{l} \left(\euler \circ u \right) \right) \nonumber\\
&= \sum^{3}_{l=1} (\mu_l \circ u)~ \nabla^{\ms{A}, \psi}\left(I_{l} \left(\euler \circ u \right) \right) \\
&= \sum^{3}_{l=1} (\mu_l \circ u)~ I_{l}(D_{\ms{A}} u)\\
&= 0
}
Note that we have used Lemma \ref{linearized cov. der} in the third step.
Substituting in \eqref{eqtn. 1} we get 
\eq{
\label{eqtn. 3}
\left(D_{\phi}\omega \right)\bullet \euler\circ u = 0
}

We can think of $D_{\phi}\omega$ as a one-form with values in the space of skew-adjoint, traceless endomorphisms of $W^+ \oplus W^-$ 
 \eqst{
 T_x X \ni v \longmapsto D_{\phi}\omega(v)|_x \bullet (\cdot)
 }
 Since, from \eqref{eqtn. 3}, one of the eigenvalues of $D_{\phi} \omega|_x$ is zero, it follows that the other one is zero as well. This is true for all $x \in X$. Thus we get $D_{\phi} \omega = 0$ or alternatively, interpreting $\omega \in \Gamma(X, \Lambda^2_+(X))$, we have $\nabla \omega = 0$.
\subsection*{Complex structure}
If $\beta$ is some non-degenerate, self-dual 2-from on $X$ such that $D_{\phi}\beta = 0$, then one can suitably modify the metric, from $g_{\sst X} \leadsto g'_{\sst X}$ so that $|\beta| = \sqrt{2}$. Therefore if $\mf{s}\Lambda^2_+(X)$ (\emph{a.k.a twistor space}) denotes the sphere bundle in $\Lambda^2_+(X)$, then $\beta \in \Gamma(X, (\mf{s}\Lambda^2_+(X))')$ and determines an almost-complex structure $J_{\beta}$ on $X$. As the Levi-Civita connection remains unchanged, we get $D_{\phi}\beta = 0$ and therefore $J_{\beta}$ is integrable. Thus $\omega$ defines a K\"ahler structure on $X$, compatible with the metric $g'_{\sst X}$. 
\end{proof}


\section{K\"ahler and symplectic structures for case of hyperK\"ahler reduction}
\label{sec:case of HK reduction}

In this section, we consider the case where $M$ is a hyperK\"ahler reduction of some flat-space $\HH^{n}$. For such hyperK\"ahler manifolds, it is possible to construct both the K\"ahler and symplectic structures on $X$. The idea here involves \emph{lifting} the Dirac equation suitably for maps taking values in $\HH^{n}$.

Let $H \subset Sp(n)$ be a compact, unitary group, acting hyper-Hamiltonianly on $\HH^{n}$ and $\mu_H: \HH^{n} \lra \mf{sp}(1)^{\ast}\otimes \mf{h}^{\ast}$ denote the associated moment map. Assume also that $H$ acts freely and properly on $\mu^{-1}_H(0)$ and let $M_H:=\mu^{-1}_H(0)/H$. If $Sp(1)$-preserves $\mu^{-1}_H(0)$, then, the permuting action of $Sp(1)$ on $\HH^n$ descends to a permuting action on $M_H$. Given a hyper-Hamiltonian $U(1)$-action on $\HH^{n}$, that commutes with the $H$-action and preserves $\mu^{-1}_H(0)$, it descends to a hyper-Hamiltonian action on $M_H$. Consider the following diagram
\eq{
\label{diag:lift of spinor}
\begin{tikzpicture}[->, node distance=2.5cm, auto, shorten >=1pt, every edge/.style={font=\footnotesize, draw}, fill=blue]
         \node (Hn+1)     {$Q$};
         \node (On+1) [right of=Hn+1]    {$M_H$};
         \node (HPn)  [above of=Hn+1]    {$P_{\widehat{H}}$};
         \node (Xn-1) [right of=HPn]     {$P \subset \HH^{n}$};
         \node (X)    [below of=Hn+1]    {$X$};

         \draw        (Hn+1) -- node [left, midway] {$\pi$} (X);        
         \draw        (Hn+1) -- node [above, midway] {$u$} (On+1);
         \draw        (HPn)  -- node [left, midway] {$\pi_{1}$} (Hn+1);
         \draw        (Xn-1) -- node [right, midway] {$\pi_{2}$} (On+1);
         \draw        (HPn)  -- node [above] {$\widehat{u}$} (Xn-1); 
         
         \end{tikzpicture}
}
Here $\pi_1$ is a $Spin^c$-equivariant submersion, $P:= \mu^{-1}_H(0)$ is the principal $H$-bundle over $M_H$. Note that $P_{\widehat{H}} \lra X$ which is a principal $H$-bundle over $Q$. Let $\widehat{u}$ be a smooth map. Define $u: Q \longrightarrow M_H$ as 
\eqst{
u(q) = \pi_2 (\widehat{u}(p)), ~~ q\in Q, ~~ p \in \pi^{-1}_1(q)
}

Clearly, the diagram commutes. On the other hand, given a smooth spinor $u: Q \lra M_H$, it defines a principal $H$-bundle over $Q$ via pull-back of $P$. The pull-back of the canonical connection $\ms{a}$ on $P$, defined as 
\eqst{
K^{P, H}_{\ms{a}}|_{p}(v) = -\pr^{\text{im}K^{P, H}}(v), ~~ v\in T_pP
} 
by $\widehat{u}$ along with the connection $\ms{A}$ on $Q$ uniquely define a connection $\mr{A}$ on $P_{\widehat{H}}$
\eq{
\label{eq:unique connection on P_H}
\mr{A} = \pi^{\ast} \ms{A} \oplus \mr{A}_{\mf{h}} \in \Lambda^{1}\left(P_{\widehat{H}}, ~\mf{spin}(4)\oplus\mf{h}\right)^{H \times Spin^c(4)}
} 
where $\mr{A}_{\mf{h}} = \widehat{u}^{\ast}\ms{a} - \langle \pi^{\ast} \ms{A}, \iota_{\mf{spin}(4)} \widehat{u}^{\ast}\ms{a}  \rangle$.
\begin{prop}
\label{prop:1-1 correspondence}
Then, there is a 1-1 correspondence between
\eq{
\label{eq:1-1 correspondence harmonic spinors}
\{(\widehat{u}, \mr{A})~|~ \Dd_{\mr{A}}\widehat{u} = 0, ~~ \mu_H \circ \widehat{u} = 0\}~~~\text{and}~~~\{(u, \ms{A})~|~ \Dd_{\ms{A}}u = 0\}
}
where, $u$ is the projection of $\widehat{u}$ to $Q$, and also between 
\eq{
\label{eq:1-1 correspondence cov. const. spinors}
\{(\widehat{u}, \mr{A})~|~D_{\mr{A}}\widehat{u} = 0, ~~ \mu_H \circ \widehat{u} = 0\}~~~\text{and}~~~\{(u, \ms{A})~|~D_{\ms{A}}u = 0\}
}
\end{prop}

\begin{proof}
To begin with, note that the condition $\mu_H \circ \widehat{u} = 0$ means that the map $\widehat{u}$ is non-vanishing, since the action of $H$ on $\mu^{-1}_H (0)$ is free.

For $h\in \HH^{n}$, define $\Hh_h := \ker d\mu_H(h) \cap (\im K^{\HH^{n}, H})^{\perp}$. It is easy to see that if $\mu_H(h) = 0$, then $\Hh_h$ is just the horizontal subspace over $h \in \mu^{-1}_G(0)$ w.r.t the canonical connection $\ms{a}$. 

We will prove the proposition in three steps. In what follows, we shall denote the $H$ and $Spin^c$-components of $\mr{A}$ by $\mr{A}_{\mf{h}}$ and $\widehat{\ms{A}}$ respectively.

\vspace{0.3cm}
\paragraph{\bf Step 1:}
\label{para:step 1}
In the first step we will prove that $\I_{\xi}D_{\mr{A}}\widehat{u}(v) \in \Hh_{\widehat{u}}$ for every $\xi \in \mf{sp}(1)$ and $v \in \Hh_{\mr{A}} \subset T P_{\widehat{H}}$. Indeed, if $\mu_H \circ \widehat{u} = 0$, then $d\widehat{u}(v) \in \ker d\mu_H(\widehat{u}(p))$. Also, $K^{\HH^{n}, H}_{\mr{A}_{\mf{h}}}|_{\widehat{u}} \in \ker d\mu_H(\widehat{u}(p))$ and $K^{\HH^{n}, Spin^c}_{\widehat{\ms{A}}}|_{\widehat{u}} \in \ker d\mu_H(\widehat{u}(p))$. Therefore, $D_{\mr{A}}\widehat{u}(v) \in \ker d\mu_H(\widehat{u}(p))$. Consequently
\eqst{
0
=\langle d\mu_H (D_{\mr{A}}\widehat{u}(v)), \xi \otimes \eta\rangle 
= \langle \I_{\xi} K^{\HH^{n}, H}_{\eta}|_{\widehat{u}(p)}, D_{\mr{A}}\widehat{u}(v) \rangle
= - \langle K^{\HH^{n}, H}_{\eta}|_{\widehat{u}(p)}, \I_{\xi} D_{\mr{A}}\widehat{u}(v) \rangle
}
for $\xi \in \mf{sp}(1), ~ \eta \in \mf{h}$. In other words, $\I_{\xi} D_{\mr{A}}\widehat{u}(v) \in (\im K^{\HH^{n}, H})^{\perp}$ for all $\xi \in \mf{sp}(1)$. Also, for $\xi' \in \mf{sp}(1)$,
\eqst{
\langle d\mu_G (\I_{\xi'}D_{\mr{A}}\widehat{u}(v)), \xi \otimes \eta\rangle = \langle d\mu_G (D_{\mr{A}}\widehat{u}(v)), [\xi', \xi] \otimes \eta\rangle = 0
}
which implies $\I_{\xi} D_{\mr{A}}\widehat{u}(v) \in \ker d\mu_G(\widehat{u}(p))$ for all $\xi \in \mf{sp}(1)$. Therefore, $\I_{\xi} D_{\mr{A}}\widehat{u}(v) \in \Hh_{\widehat{u}}$.

\vspace{0.3cm}

\paragraph{\bf Step 2:}
\label{para:step 2}

If $\Dd_{\mr{A}} \widehat{u} = 0$, then from \eqref{eq:gen. dirac operator explicit}, we have
\eqst{
0 = D_{\mr{A}}\widehat{u}(\tilde{e_0}) - \sum^3_{i=1} I_{i}D_{\mr{A}} \widehat{u}(\tilde{e_i})
}

From Step 1, we get $D_{\mr{A}}\widehat{u}(\tilde{e_0}) \in \Hh_{\widehat{u}}$. It follows that $D_{\mr{A}}\widehat{u}(\tilde{e_i}) \in \Hh_{\widehat{u}}$ for all $i=1,2,3$. Consequently, for any $v \in \Hh_{\mr{A}}, ~~\pr^{\im K^{\HH^{n}, H}}D_{\mr{A}}\widehat{u}(v) = 0$ and we get $\displaystyle K^{\HH^{n}, H}_{\mr{A}_{\mf{h}}(v)} = - \pr^{\im K^{\HH^{n}, H}} d\widehat{u}(v)$. In other words, the $\mf{h}$-connection component of $\mr{A}$ is just the pull-back of the canonical connection on the Riemannian submersion $\mu_H^{-1}(0)$.

Since the diagram commutes, $d\pi_2(D_{\mr{A}}\widehat{u}) = D_{\ms{A}}u$. Also, as $D_{\mr{A}}\widehat{u}(\tilde{e_i}) \in \Hh_{\widehat{u}}$ for all $i=0, 1,2,3$
\alst{
0
&=d\pi_2 (\Dd_{\mr{A}}\widehat{u})\\
&= d\pi_2 \left(D_{\mr{A}}\widehat{u}(\tilde{e_0}) - \sum^3_{i=1} \iota^{\ast}I_{i}~D_{\mr{A}} \widehat{u}(\tilde{e_i})\right)\\
&= d\pi_2 (D_{\mr{A}}\widehat{u}(\tilde{e_0})) - \sum^3_{i=1} \pi^{\ast}_2 \tilde{I_{i}}~ d\pi_2(D_{\mr{A}} \widehat{u}(\tilde{e_i}))\\
&= D_{\ms{A}}u(\tilde{e_0}) - \sum^3_{i=1} \tilde{I_{i}}~ D_{\ms{A}}u (\tilde{e_i})\\
&= \Dd_{\ms{A}}u
}
Thus, $\Dd_{\mr{A}}\widehat{u} = 0$ implies $\Dd_{\ms{A}}u = 0$. On the other hand if $\displaystyle K^{\HH^{n}, H}_{\mr{A}_{\mf{h}}(v)} = - \pr^{\im K^{\HH^{n}, H}} d\widehat{u}(v)$, then $D_{\mr{A}}\widehat{u} \in \Hh_{\widehat{u}}$ and so $d\pi_2 (\Dd_{\mr{A}}\widehat{u}) = \Dd_{\ms{A}}u$. Therefore, if $\Dd_{\ms{A}}u = 0$, it implies that $\Dd_{\mr{A}}\widehat{u} \in \im K^{\HH^{n}, H}$. But since, 
\eqst{
\Dd_{\mr{A}}\widehat{u}
= D_{\mr{A}}\widehat{u}(\tilde{e_0}) - \sum^3_{i=1} \pi^{\ast}_2 I_{i}~D_{\mr{A}} \widehat{u}(\tilde{e_i}) \in \Hh_{\widehat{u}}
}
it follows that $\Dd_{\mr{A}}\widehat{u} \in (\im K^{\HH^{n}, H})^{\perp}$ and so $\Dd_{\mr{A}}\widehat{u} = 0$.

\vspace{0.3cm}

\paragraph{\bf Step 3:}
\label{para:step 3}
Using an argument verbatim to the one in Step 2 above, we can prove that there is a 1-1 correspondence between
\eq{
\{D_{\mr{A}}\widehat{u} = 0, ~~ \mu_H \circ \widehat{u} = 0\}~~~\text{and}~~~\{D_{\ms{A}}u = 0\}
}

This proves the statement.

\end{proof}


\subsection*{K\"ahler and symplectic structures}

Proposition \ref{prop:1-1 correspondence} allows us to use the tools for the usual $Spin$-Dirac operator. 

For the rest of the section, assume that the $\mf{h}$-component of a connection $\mr{A}$ is given by $\mr{A}_{\mf{h}}$ as in \eqref{eq:unique connection on P_H}.

\begin{lem}
\label{lem:existence of connection}
Given a non-vanishing spinor $\widehat{u}: P_{\widehat{H}} \lra \HH^{n}$, there exists a unique connection $\mr{A}$ such that $\Dd_{\mr{A}}\widehat{u} = 0$.
\end{lem}
\begin{proof}
This is a consequence of the fact that since $\widehat{u}$ is non-vanishing, Clifford multiplication is modelled on quaternionic multiplication. Since the $\mf{h}$-component of the connection is given by the pull-back of the canonical connection on Riemannian submersion $\mu^{-1}_H(0)$, the only variable is the $U(1)$-conection component.  

Let $\mr{A}$ be the lift of a connection $\ms{A}$ on $Q$, determined by a $U(1)$-connection $A$. For a connection $\mr{A}' = \mr{A} + i\alpha$, $\alpha \in T^{\ast}X$, we have $\Dd_{\mr{A} + i\alpha}\widehat{u} = \Dd_{\mr{A}}\widehat{u} + i \alpha\bullet \widehat{u}$. To solve $\Dd_{\mr{A} + i\alpha}\widehat{u} = 0$, we need to solve $\Dd_{\mr{A}}\widehat{u} = -i \alpha\bullet \widehat{u}$.

Consider the map
\eqst{
\R^{4} \otimes \HH^{n^{\ast}} \lra \HH^{n^{\ast}}, ~~ a \otimes h \longmapsto a\bullet h = h\cdot \bar{a}
}

Fix $h_0 \in \HH^{n^{\ast}}$ and let $h'\in \HH^{n^{\ast}}$. Define $a := \displaystyle \frac{\bar{h'}^t\cdot h_0}{|h_0|^2} \in \HH \cong \R^4$. Then, 
\eqst{
a\bullet h_0 = h_0 \cdot \frac{\bar{h}_0^t\cdot h'}{|h_0|^2} = h'
}
Thus, for a a fixed $h_0$ and any $h'$, there exists an $a \in \R^4$ such that $a \bullet h_0 = h'$. In other words, $\Dd_{\mr{A}}\widehat{u} = -i \alpha\bullet \widehat{u}$ always has a solution. The uniqueness of the connection $\mr{A}$ follows from the unique continuation property of the Dirac operator.

\end{proof}

Observe that if we change the $U(1)$-connection, then $D_{\mr{A}'}\widehat{u} = D_{\mr{A}}\widehat{u} + i\alpha\bullet\widehat{u}$. In other words, the covariant derivative changes along the direction of the $U(1)$-orbit of $\widehat{u}$.
\begin{prop}\cite[Corollary 3.11]{scorpan2002}
\label{prop:covariant derivative formula}
\eqst{
\|\widehat{u}\|^2 \cdot D_{\mr{A}}\widehat{u} = i (D_{\phi}\omega) \bullet \widehat{u} + \left\langle D_{\mr{A}}\widehat{u},~ i\widehat{u}\right\rangle_{\R} \cdot i\widehat{u}
}
Consequently, if $D_{\mr{A}}\widehat{u} = 0$, then $\omega$ defines a K\"ahler structure on the base manifold, compatible with a metric which is a scalar multiple of $g_{\sst X}$.
\end{prop}
\begin{proof}
The hyperK\"ahler $U(1)$-moment map is given by 
\eqst{
\mu:\HH^{n} \lra \mf{sp}(1)^{\ast}, ~~ \mu(h) = \frac{1}{2}\bar{h}^tih
}
Therefore $\omega\bullet \widehat{u} = (\mu \circ \widehat{u})\bullet \widehat{u} = \displaystyle -\frac{i}{2}\|\widehat{u}\|^2\cdot \widehat{u}$. The identity now follows from Corollary 3.11 of \cite{scorpan2002}.

\end{proof}

Composed with Clifford multiplication, we obtain the \emph{splitting formula for Dirac operator}

\begin{prop}\cite[Theorem 3.15]{scorpan2002}
\label{prop:Dirac operator splitting formula}
Given a $H \times Spin^c$-structure $P_{\widehat{H}}$ and a connection $\mr{A}$ on $P_{\widehat{H}}$, we have
\eq{
\|\widehat{u}\|^2 \Dd_{\mr{A}}\widehat{u} = i \left( 2d^{\ast}\omega + \langle D_{\mr{A}}\widehat{u}, i\widehat{u}\rangle_{\R} \right) \bullet \widehat{u}
}
Therefore, if $\langle D_{\mr{A}}\widehat{u}, i\widehat{u}\rangle_{\R} = 0$, then every harmonic spinor $\widehat{u}$ defines a symplectic structure on the base manifold, compatible with a metric conformal to $g_{\sst X}$.
\end{prop}

\begin{note}
Using the existence argument as in Lemma \ref{lem:existence of connection}, one can prove that there exists a unique connection $\mr{A}'$ such that $\langle D_{\mr{A}'}\widehat{u}, i\widehat{u}\rangle_{\R} = 0$
\end{note}

\begin{rmk}
It is worth noting that every non-vanishing, covariantly constant/ harmonic spinor defines the same K\"ahler/symplectic structure. Indeed, one can think of $\mu$ as a quadratic map
\eq{
\label{eq:quadratic map between spheres}
\mu: S^{4n-1} \lra S^2
}
where $S^{4n-1}$ is the unit sphere in $\HH^{n}$ and $S^2$ is the unit sphere in $\Lambda^2_+(\R^4)$. For $m_1\geq 2m_2$, every polynomial map from $S^{m_1}$ to $S^{m_2}$ is a constant map \cite{wood1968}. Therefore, for $n>1$, $\mu: S^{4n-1} \lra S^2$, $\mu$ is a constant map. 

This has the following consequence: let $\mr{A}'$ be the unique connection such that $\langle D_{\mr{A}'}\widehat{u}, i\widehat{u}\rangle_{\R} = 0$. Then, from Proposition \ref{prop:covariant derivative formula} and Theorem \ref{prop:Dirac operator splitting formula} we conclude 

\begin{lem}
\label{lem:kahler structure hk-reduction}
If there exists a non-vanishing spinor $\widehat{u}$ which is covariantly constant with respect to the unique connection $\mr{A}'$ such that $\langle D_{\mr{A}'}\widehat{u}, i\widehat{u}\rangle_{\R} = 0$, then every non-vanishing spinor is covariantly constant w.r.t a unique connection and the defines the same complex structure.

\end{lem}

Similarly, every harmonic spinor, satisfying $\langle D_{\mr{A}}\widehat{u}, i\widehat{u} \rangle_{\R} = 0$ defines the same symplectic structure on $X$.
\end{rmk}


 \begin{rmk}
 
  One can repeat the entire argument in Theorems \ref{thm:Swann bundles kahler structure}, \ref{thm:HK-reduction kahler structure} and \ref{thm:HK-reduction symplectic structure} for $G=U(n)$ instead of $U(1)$. One only needs to observe that the $U(n)$-moment map splits into 2 components \[\mu_{\mf{u}(n)} = \mu_{\mf{su}(n)} + \mu_{\mf{u}(1)}\] and $\mf{u}(1)$-component defines a self-dual 2-form on $X$.
      
\end{rmk}
 

\section{Equivalence of \texorpdfstring{$Spin^c-$}~structures} 
\label{sec:equivalence of structures}
 Let $(X, g_{\sst X})$ be an oriented, compact, 4-dimensional manifold and let $Q \lra X$ be a fixed $Spin^c$-structure. A non-degenerate, self-dual 2-form $\beta \in \Gamma(X, \mf{s}\Lambda^2_+(X))$ defines an almost-complex structure $J$ on $X$, which is compatible with a metric in the conformal class of $g_{\sst X}$. Then $(\beta, J)$ define a $Spin^c$-structure $Q_{\beta} \lra X$.
 
 \begin{prop}
 \label{thm:equivalence of structures}
  Let $M$ be as in Theorem \ref{thm:Swann bundles kahler structure} or Theorem \ref{thm:HK-reduction kahler structure} and $u \in \Ss$ be a spinor whose range does not contain a fixed point of the $U(1)$-action. Let $A$ be a connection on $P_{U(1)}$ and $\ms{A}$ be the induced connection on $Q$. If $D_{\ms{A}} u = 0$, then the $Spin^c$-structure $Q_{\omega}$ is isomorphic to $Q$.
 \end{prop}
 
 \begin{proof}
The condition $D_{\ms{A}}u = 0$ implies that $D_{\phi}\omega= 0$, or in other words $d(\mu\circ u) = 0$ and thus $\omega$ has a constant length. Modifying the metric $g_{\sst X}$ to a metric $g'_{\sst X} := c\cdot g_{\sst X}$ for some $c \in \R$, we may assume  that $|\omega| = \sqrt{2}$. The metric $g'_{\sst X}$, is thus a K\"ahler metric defining the $Spin^c$-structure $Q_{\omega}$.
 
 Any metric $g''_{\sst X}:= e^{2f}g_{\sst X}$ in the conformal class of $g_{\sst}$ induces an isomorphism between the respective frame bundles
 \eq{
 P_{SO(4)} \xrightarrow{e^{-\pi^{\ast}f}} P''_{SO(4)}
 }
 
 where $\pi: P_{SO(4)} \lra X$. Any point $p \in P$ is a linear isomorphism $p:\R^4 \lra T_{\pi(p)}X$. Consider $p \in P_{SO(4)}$ and $p''\in P''_{SO(4)}$ such that $\pi(p) = \pi''(p'') = x$. Then, $p\circ (p'')^{-1}: \R^4 \lra \R^4$ is an automorphism. But $p'' = e^{-f(x)}p$. Therefore, $p\circ (p'')^{-1}$ is an isomorphism of $\R^4$ obtained by scalar multiplication by $e^{-f(x)}$. If $e^{-f}$ is constant, say $c$, then the isomorphism is independent of $x\in X$. In other words, we get an isomorphism between $(\R^4, g_{\R^4})$ and $(\R^4, c\cdot g_{\R^4})$, where $g_{\R^4}$ is the standard metric on $\R^4$. This induces an isomorphism of the respective complexified Clifford algebras $\gamma: \mc{C}l_4 \otimes \C \lra \mc{C}l''_4\otimes \C$ which preserves the positive elements and therefore, induces an isomorphism of the respective $Spin^c$ groups $\gamma: Spin^c(4) \lra (Spin^c(4))''$. Therefore, if $g'_{\sst X} := c\cdot g_{\sst X}$ and $Q$ and $Q'$ denote the respective $Spin^c$-bundles over $X$, then the map $\gamma$ induces an isomorphism of the bundles $Q$ and $Q'$. In conclusion, the $Spin^c$-structure $Q_{\omega}$ is isomorphic to $Q$.
 
\end{proof}


\section{Constant Scalar curvature}
\label{sec:const. scalar curvature}
In this section, we show that if the solution space to equations \eqref{eq:gen. seiberg-witten} contains a pair $(u, \mr{A})$ such that $D_{\mr{A}}u = 0$ then the scalar curvature of the base 4-dimensional manifold is necessarily (negative) constant.

Let $M$ be as in Theorem \ref{thm:Swann bundles kahler structure} or Theorem \ref{thm:HK-reduction kahler structure}. If a pair $(u, \ms{A})$ satisfies $D_{\ms{A}}u = 0$, we get $\Dd_{\ms{A}} u = 0$ and the Weitzenb\"ock formula \eqref{weitzenbock formula} gives
\eqst{
0 = \frac{s_{\sst X}}{4} \euler\circ u + \mc{Y}(F^+_{A})|_{u}
}
 Take the inner product on both the sides, with $\euler \circ u$ to get 
\eq{
\label{weitenbock with cov. const. spinor}
 0 = \frac{s_{\sst X}}{4}\cdot \rho_0 \circ u + \langle \mu\circ u, F^+_A \rangle
}

The second term in the above expression is computed as follows:
 \alst{
 \langle \mc{Y}(F^+_{A})|_{u}, \euler\circ u\rangle
 &= \sum^3_{l=1} \left \langle I_l K^M_{\langle F^+_{A}, \xi_l \rangle}|_{u}, I_{l}K^M_{\xi_l}|_{u} \right\rangle = \sum^3_{l=1} \left \langle K^M_{\langle F^+_{A}, \xi_l \rangle}|_{u}, K^M_{\xi_l}|_{u} \right\rangle \\ &= \sum^3_{l=1} \mu(\xi_l \otimes \langle F^+_{A}, \xi_l \rangle)\circ u = \langle \mu\circ u, F^+_A \rangle
}

Since $D_{\ms{A}} u = 0$, $d(\rho_0\circ u) = \langle d\rho_0, D_{\ms{A}}u \rangle = 0$,  $\rho_0 \circ u = c$, a constant. Therefore,
 \eqst{
 -\frac{s_{\sst X}}{4}\cdot \rho_0 \circ u
 = \langle \mu\circ u, F^+_A \rangle \\
 }

If the pair $(u, \ms{A})$ is a solution to generalized Seiberg-Witten equations \eqref{eq:gen. seiberg-witten}, then 
\eqst{
 -\frac{s_{\sst X}}{4}\cdot \rho_0 \circ u
 = |\mu\circ u|^2 \\
 }
Since $\rho_0 \circ u$ is a positive constant, this implies $s_{\sst X}$ is a negative constant. Therefore, the metric on $X$ is a metric of constant scalar curvature. Since the metric defined by the K\"ahler structure is a scalar multiple of $g_{\sst X}$, it follows that the K\"ahler structure defined by $\Phi(\mu\circ u)$ is of constant scalar curvature.

 \appendix
 
 \section{Vector bundles and connections}
 \label{appendix:VB and connections}
 Let $\pi_E: E \lra X$ be a vector bundle. Then consider $T\pi_E: TE \lra TX$. Then $\mc{V}_E \subset\ker(T\pi_E) \subset TE$ is called the \emph{verticle sub-bundle}. A connection on $E$ is a choice of a smooth \emph{horizontal sub-bundle} $\mc{H}_E$ such that $TE = \mc{H}_E \oplus \mc{V}_E$. Denote by $\pr_{\sst \mc{V}}$ and $\pr_{\sst \mc{H}_E}$ the projections to the verticle and the horizontal sub-bundles respectively. A connection on $E$ is said to be linear if $\pr_{\sst \mc{V}_E}$ is linear w.r.t $T\pi_E$ 
 
 \subsection*{Verticle lift}
 Consider the pull-back bundle $E \times_M E$. Then, the map 
 \eqst{
 \vl_E: E\times_M E \lra \mc{V}_E, ~~ (v,w) \longmapsto \frac{d}{dt}(v + tw)|_{t=0}
 }
 gives an isomorphism of vector bundle over $E$ and is called a \emph{vertical lift}.
 
 \subsection*{Connector}
 A \emph{connector} is a smooth map $K: TE \lra E$ is a smooth map that satisfies $K \circ \vl_{E} = \pr_2: E\times_M E \lra E$ and is a vector bundle homomorphism for both the vector bundle structures on $E$; i.e $T\pi_E: TE \lra TX$ and $\pi: TE \lra E$.
 
 Given a linear connection $\Phi: TE \lra TE$, its connector is given by the composition
 \begin{center}
 \label{connector} 
  \begin{tikzpicture}[%
    >=stealth,
    shorten >=2pt,
    shorten <=2pt,
    auto,
    node distance=2.5cm,
    text centered
  ]
    \node (a) {$TE$};
    \node (b) [right of=a] {$\mc{V}_E$};
    \node (c) [right of=b] {$E\times_X E$};
    \node (d) [right of=c] {$E$};

    \path[->] (a) edge     node      (phi)     {$\Phi$}             (b)
              (b) edge     node      (vl_E)    {$(\vl_{E})^{-1}$}   (c)
              (c) edge     node      (pr_2)    {$\pr_{2}$}          (d);
  \end{tikzpicture}
  
 \end{center}
 
 A connector on $E$ induces a covariant derivative
 \alst{
  & \nabla^K: \Gamma(X, E) \lra \Gamma(X, T^{\ast}X\otimes E)\\
  & \nabla^K_v(s) = K(Ts(v))~~\text{for}~~ v\in TX, ~~ s\in \Gamma(X, E)
 }
 
 \section{Principal bundles and covariant derivatives on associated bundles}
 \label{principal bundles and covariant derivatives on associated bundles}
Let $\pi_P:P \lra X$ be a principal $G$-bundle and $M$ be a manifold with a smooth action of $G$.

\begin{thm}\cite[Satz 3.5]{baum}
\label{section-map equivalence}
  There is a bijection between the space of $G$-equivariant maps from $P \rightarrow M$ and the sections of the fibre bundle $\mc{F}:= P\times_G M \lra X$
 \al{
 & C^{\infty}(P,M)^G \lra \Gamma(X, \mc{F})\nonumber \\
 & u \mapsto s_u,~~ \text{where} ~~ s_u (x) = [p, u(p)] ~~\text{and}~~ \pi_P(p) = x
 }
 \end{thm}
 
 \subsection*{Covariant derivatives on asociated fibre-bundles}
Let $A \in \Lambda^1(P, \mf{g})^G$ be a connection on $P$. We define the covariant derivative of $u \in \Map(P, M)^{G}$ w.r.t as
 \alst{
 \label{covariant derivative}
 D_A u = Tu + K^M_{A}|_u \in \Hom(TP, TM)_{hor}
 }
 where $K^M_{A}|_{u}: TP \rightarrow u^{\ast}TM$ is vector bundle homomorphism and the subscript ``hor" implies that $D_A u$ vanishes for all $v \in \mc{V}_{TP}$
\eqst{
K^M_{A}|_{u} (v) = K^M_{A(v)}|_{u(p)}, ~~~ v\in T_p P
}
 \begin{prop}\cite[Theorem 42.1]{KM}
 \label{frechet manifold}
  Let $X$ be a compact manifold. Then the space $\Map(P, M)^{G}$ is a Frech\'et manifold modelled on topological vector spaces:
  \eqst{
  T_u\Map(P,M)^G = \Gamma(P, u^{\ast}TM)^G
  }
 \end{prop}

 Therfore the covariant derivative can be interpreted as a section of the infinite-dimensional Frech\'et vector-bundle:
 \begin{center}
\begin{tikzpicture}[%
    >=stealth,
    shorten >=2pt,
    shorten <=2pt,
    auto,
    node distance=5cm,
    text centered
  ]
    \node (a) {$\Hom(\mc{H}_A, TM)^G$};
    \node (b) [right of=a] {$\Map(P, M)^G$};

    \path[->] (a) edge                 node [below, midway]     (pi)   {$\Pi$}   (b)
              (b) edge [bend right=20] node [above, midway]       (dirac)   {$D_A$}   (a);
  \end{tikzpicture}

\end{center}

Note that 
 \eqst{
 T\Hom(\mc{H}_A, TM)^G = \Hom(\mc{H}_A, TTM)^G, ~~ \mc{V}\Hom(\mc{H}_A, TM)^G = \Hom(\mc{H}_A, \mc{V}TM)^G
 } 
  A connector $\psi: TTM \lra TM$ induces a connector $\Psi$ on $\Hom(\mc{H}_A, TM)^G$. Using this, one can compute the linearization of $\nabla^{\Psi}D_{A}$.

 \begin{lem}\cite[Section 2.4]{henrik}
  \label{linearization cov. der.}
  The linearization $\nabla^{\Psi}D_{A}$ of the covariant derivative coincides with the first-order differential operator 
  \eq{
\nabla^{A, \psi} :~ C^{\infty}(P, TM)^{G} \lra \Hom(TP, TM)^{G}_{hor}, ~~ v \longmapsto \psi \circ Tv \circ \pr_{\Hh_{A}}
 }
 \end{lem}


\bibliographystyle{alpha}
\bibliography{references}

\begin{thebibliography}{ADHM78}

\bibitem[ADHM78]{adhm}
M.~F. Atiyah, V.~G. Drinfel\'d, N.~J. Hitchin, and Y.~I. Manin.
\newblock Construction of instantons.
\newblock {\em Phys. Lett. A}, 65(3):185--187, 1978.

\bibitem[Bau09]{baum}
H.~Baum.
\newblock {\em Eichfeldtheorie. Eine Einf\"uhrung in die Differentialgeometrie
  auf Faserb\"undeln}.
\newblock Springer-Verlag, Berlin, 2009.

\bibitem[BLPR00]{BLPR2000}
Y.~Byun, Y.~Lee, J.~Park, and J.~S. Ryu.
\newblock Constructing the {K}\"ahler and the {S}ymplectic structures from
  certain spinors on 4-manifolds.
\newblock {\em Proc. Amer. Math. Soc.}, 129(4):1161--1168, 2000.

\bibitem[HKLR87]{hklr}
N.~J. Hitchin, A.~Karlhede, U.~Lindstr\"om, and M.~Ro\v{c}ek.
\newblock {H}yperk\"ahler metric and {S}upersymmetry.
\newblock {\em Communications in Mathematical Physics}, 108(4):535--589, 1987.

\bibitem[KM97]{KM}
A.~Kriegl and P.~Michor.
\newblock {\em {T}he Convenient Setting of Global Analysis}, volume~53 of {\em
  Mathematical Surveys and Monographs}.
\newblock American Mathematical Society, Providence, RI, 1997.

\bibitem[Kro88]{kron88}
P.~B. Kronheimer.
\newblock A hyperk\"ahler structure on the cotangent bundle of a complex {L}ie
  group.
\newblock \url{http://www.math.harvard.edu/~kronheim/hkgc.pdf}, 1988.
\newblock Pre-print.

\bibitem[Kro90]{kronheimer}
P.~B. Kronheimer.
\newblock A hyper-{K}\"ahlerian structure on coadjoint orbits of a {S}emisimple
  complex group.
\newblock {\em Journal of the London Mathematical Society}, s2-42(2):193--208,
  October 01 1990.

\bibitem[KS96]{kobak-swann96}
P.~Z. Kobak and A.~Swann.
\newblock Classical nil-potent orbits as hyperk\"ahler quotients.
\newblock {\em Int. J. Math.}, 07(02):193--210, 1996.

\bibitem[Pid04]{victor}
V.~Ya. Pidstrygach.
\newblock Hyper{K}\"ahler manifolds and {S}eiberg-{W}itten equations.
\newblock {\em Proc. Steklov Inst. Math.}, pages 249--262, 2004.

\bibitem[Sch10]{henrik}
H.~Schumacher.
\newblock {G}eneralized {S}eiberg-{W}itten equations: {S}wann bundles and
  ${L}^{\infty}$-estimates.
\newblock Master's thesis, Mathematisches Institut, Georg-August-Universit\"at,
  G\"ottingen, \url{http://www.uni-math.gwdg.de/preprint/mg.2010.02.pdf}, 2010.

\bibitem[Sco02]{scorpan2002}
A.~Scorpan.
\newblock {N}owhere zero harmonic spinors and self-dual 2-forms.
\newblock {\em Commun. Contemp. Math}, 4(1):45--64, 2002.

\bibitem[Swa91]{swann}
A.~Swann.
\newblock Hyper{K}\"ahler and {Q}uaternionic {K}\"ahler geometry.
\newblock {\em Math. Ann.}, 3:421--450, 1991.

\bibitem[Tau99]{taubes}
C.~H. Taubes.
\newblock Nonlinear {G}eneralizations of a 3-manifold's {D}irac operator.
\newblock In {\em Trends in mathematical physics (Knoxville, TN, 1998), AMS/IP
  Stud. Adv. Math.}, volume~13, pages 475--486. Amer. Math. Soc., Providence,
  RI, 1999.

\bibitem[Woo68]{wood1968}
R.~Wood.
\newblock {P}olynomial maps from spheres to spheres.
\newblock {\em Invent. Math.}, 5(3):163--168, 1968.

\end{thebibliography}

\end{document}